\documentclass[12pt]{amsart}

\makeindex

\usepackage[all, cmtip]{xy}
\usepackage{amssymb}
\usepackage{amsmath}
\usepackage{amsthm}
\usepackage{amscd}
\usepackage{amsfonts}
\usepackage{amssymb}
\usepackage[pdftex]{graphicx}
\usepackage{multirow}
\usepackage[usenames,dvipsnames,svgnames,table]{xcolor}
\usepackage{graphicx}
\usepackage{enumerate}

\usepackage[bookmarks=true,bookmarksnumbered=true,
 pdftitle={},
 pdfsubject={},
 pdfauthor={},
 pdfkeywords={TeX, dvipdfmx, hyperref, color,},
 colorlinks=true, linkcolor=RoyalBlue, citecolor=Emerald]
 {hyperref}

\theoremstyle{definition}

\theoremstyle{remark}

\theoremstyle{corollary}

\theoremstyle{theorem}

\theoremstyle{corollary}

\newtheorem{theorem}{Theorem}[section]
\newtheorem{lemma}[theorem]{Lemma}
\newtheorem{proposition}[theorem]{Proposition}

\theoremstyle{corollary}
\newtheorem{corollary}[theorem]{Corollary}

\theoremstyle{definition}

\theoremstyle{remark}
\newtheorem{remark}[theorem]{Remark}

\numberwithin{equation}{section}

\newcommand{\C}{\mathbb{C}}

\newcommand{\Z}{\mathbb{Z}}

\def\P{\mathbb{P}}

\def\reg{\operatorname{reg}}

\def\gon{\operatorname{gon}}
\def\ev{\operatorname{ev}}
\def\pr{\operatorname{pr}}

\newcommand{\suchthat}{\,\ifnum\currentgrouptype=16 \middle\fi|\,}

\input xy
\xyoption{all}

\title[Smooth projective varieties with $2$-regular structure sheaf]{Classification and syzygies of smooth projective varieties with $2$-regular structure sheaf}

\begin{document}

\author{Sijong Kwak}
\address{Department of Mathematical Sciences, KAIST, Daejeon, Korea}
\email{sjkwak@kaist.ac.kr}

\author{Jinhyung Park}
\address{School of Mathematics, Korea Institute for Advanced Study, Seoul, Korea}
\email{parkjh13@kias.re.kr}

\thanks{S. Kwak was supported by Basic Science Research Program through the National Research Foundation of Korea (NRF) funded by the Ministry of Science and ICT (2015R1A2A2A01004545).}

\date{\today}



\begin{abstract}
The geometric and algebraic properties of smooth projective varieties with $1$-regular structure sheaf are well understood, and the complete classification of these varieties is a classical result. The aim of this paper is to study the next case: smooth projective varieties with $2$-regular structure sheaf. First, we give a classification of such varieties using adjunction mappings. Next, under suitable conditions, we study the syzygies of section rings of those varieties to understand the structure of
the Betti tables, and show a sharp bound for Castelnuovo-Mumford regularity.
\end{abstract}

\maketitle


\section{Introduction}

Throughout the paper, we work over the field $\C$ of complex numbers. We begin by recalling the definition of Castelnuovo-Mumford regularity.
A coherent sheaf $\mathcal{F}$ on a projective variety $X$ with a very ample line bundle $L$ is said to be \emph{$m$-regular} in the sense of Castelnuovo-Mumford if $H^i(X, \mathcal{F}\otimes L^{m-i})=0$ for all $i >0$.
By \cite[Theorem 1.8.5]{positivityI}, if $\mathcal{F}$ is $m$-regular, then $\mathcal{F}$ is $(m+1)$-regular.
An embedded projective variety $X \subsetneq \P^r$ is said to be \emph{$m$-regular} if the ideal sheaf $\mathcal{I}_{X|\P^r}$ is $m$-regular, i.e., $H^i(\P^r, \mathcal{I}_{X|\P^r}(m-i))=0$ for $i >0$. Note that $X \subsetneq \P^r$ is $m$-regular if and only if
\begin{enumerate}[\indent{\tiny$\bullet$}]
\item $X \subsetneq \P^r$ is $(m-1)$-normal, i.e., the natural restriction map
$$
H^0(\P^{r}, \mathcal{O}_{\P^{r}}(m-1)) \rightarrow H^0 (X, \mathcal{O}_X(m-1))
$$
is surjective, and
\item $\mathcal{O}_X$ is $(m-1)$-regular, i.e., $H^i(X, \mathcal{O}_X(m-1-i))=0$ for $i>0$.
\end{enumerate}
The \emph{Castelnuovo-Mumford regularity} $\reg(X)$ is the least integer $m$ such that $X \subsetneq \P^r$ is $m$-regular. We also denote by $\reg(\mathcal{O}_X)$ the least integer $m$ such that $\mathcal{O}_X$ is $m$-regular. For more detail, we refer to \cite[Section 1.8]{positivityI}.

\medskip

Now, let $X \subsetneq \P^r$ be a non-degenerate smooth projective variety of dimension $n$, and $H$ be its general hyperplane section. We always assume that $X \subsetneq \P^r$.
Then $\reg(X) \geq 2$ and $\reg(\mathcal{O}_X) \geq 1$.
It is a classical fact due to Eisenbud-Goto \cite{EG} that
\begin{equation}\tag{$\textcolor{OrangeRed}{*}$}\label{clfact}
\text{$\reg(X) = 2$ if and only if $X \subsetneq \P^r$ is a variety of minimal degree}
\end{equation}
(for a generalization to algebraic sets, see \cite{EGHP}). Recall that a variety of minimal degree is either a rational normal scroll, a quadric hypersurface $Q^n \subset \P^{n+1}$, or the second Veronese surface $v_2(\P^2) \subset \P^5$ (see e.g., \cite{EH}).
To motivate our first result and approach, we give a quick geometric proof of the classical fact (\ref{clfact}) based on adjunction theory.

\begin{proof}[Proof of the Fact $($\ref{clfact}$)$]
The direction $(\Leftarrow)$ is trivial. For the converse direction $(\Rightarrow)$, we divide into two steps.

\smallskip

\noindent\underline{Step 1}.
Here we classify smooth projective varieties $X \subsetneq \P^r$ with $\reg(\mathcal{O}_X)=1$. Note that $\reg(\mathcal{O}_X)=1$ implies $H^0(X, \mathcal{O}_X(K_X+(n-1)H)) = H^n(X, \mathcal{O}_X(-(n-1)))^*=0$. Thus $K_X+(n-1)H$ is not base point free. By a classification result in adjunction theory \cite[Theorem 1.4]{Io1}, we see that $(X, H)$ is either
$$
(\P^n, \mathcal{O}_{\P^n}(1)), (Q^n, \mathcal{O}_{\P^{n+1}}(1)|_{Q^n}), (\P^2, \mathcal{O}_{\P^2}(2)), \text{or } (\P(E), \mathcal{O}_{\P(E)}(1))
$$
where $Q \subset \P^{n+1}$ is a quadric hypersurface and $E$ is a very ample vector bundle of rank $n$ on a smooth projective curve. Since we also have $H^1(X, \mathcal{O}_X)=0$, it follows that $X \subsetneq \P^r$ is a quadric hypersurface, a (possibly projected) second Veronese surface, or a rational scroll.

\smallskip

\noindent\underline{Step 2}. $\reg(X)=2$ implies that $\reg(\mathcal{O}_X)=1$ and $X \subsetneq \P^r$ is linearly normal. So, it follows immediately that $X \subsetneq \P^r$ is a quadric hypersurface, the second Veronese surface $v_2(\P^2) \subset \P^5$, or a rational normal scroll embedded by the linear system $|\mathcal{O}_{\P(E)}(1)|$. All of these varieties are of minimal degree.
\end{proof}

Along this line, it is natural to consider the classification problem of smooth projective varieties $X \subsetneq \P^r$ with $\reg(X)=3$. The first step would also be a classification of smooth projective varieties $X\subsetneq \P^r$ with $\reg(\mathcal{O}_X)=2$ using adjunction theory.
Note that $\mathcal {O}_X$-regularity with respect to a very ample divisor $H$ is an intrinsic property not depending on the embedding of $X$ given by $H$.
Recall that 
$$
\text{ $\reg(\mathcal{O}_X)=1$ if and only if $(Q^n, \mathcal{O}_{\P^{n+1}}(1)|_{Q^n})$, $(\P^2, \mathcal{O}_{\P^2}(2))$, or $(\P(E), \mathcal{O}_{\P(E)}(1))$, }
$$
where $Q \subset \P^{n+1}$ is a quadric hypersurface and $E$ is a very ample vector bundle of rank $n$ on a projective line $\P^1$.

The first main theorem of this paper completes the first step classifying smooth projective varieties with $\reg(\mathcal{O}_X)=2$ .

\begin{theorem}\label{classthmreg2}
Let $X \subsetneq \P^r$ be a non-degenerate smooth projective variety of dimension $n$, and $H$ be its general hyperplane section. Then $\reg(\mathcal{O}_X) = 2$ if and only if $(X, H)$ is one of the following:
\begin{enumerate}[\noindent$(1)$]
\item $n=1$: $X=C$ is a smooth projective curve of genus $g \geq 1$, and $H$ is a non-special very ample divisor $(H^1(C, \mathcal{O}_C(H))=0)$.

\item $n=2$: $X=S$ is a smooth projective surface with $p_g(S)=h^0(S, \mathcal{O}_S(K_S))=0$, and $H$ is a very ample divisor with $H^1(S, \mathcal{O}_S(H))=0$, but $(S, H) \neq (\P^2, \mathcal{O}_{\P^2}(2))$.

\item $n \geq 3$: $X=\P(E)$, where $E$ is a very ample vector bundle of rank $n$ on a smooth projective curve $C$ of genus $g \geq 1$ such that $H^1(C, E)=0$, and $H$ is the tautological divisor of $\P(E)$.

\item $n \geq 3$: $X$ is a del Pezzo manifold, and $H$ is a very ample divisor such that $-K_X=(n-1)H$.

\item $n \geq 3$: $X$ is a smooth member of $|2L+\pi^* D|$, and $H:=L|_X$, where $L$ is the tautological divisor of $\P(E)$ with a natural projection $\pi \colon \P(E) \to C$ such that $E$ is a globally generated vector bundle of rank $n+1$ on a smooth projective curve $C$ with $H^1(C, E)=0$, and $D$ is a divisor on $C$. In this case, $X$ has a hyperquadric fibration over $C$.

\item $n \geq 3$: $X=\P(E)$, where $E$ is a very ample vector bundle of rank $n-1$ on a smooth projective surface $S$ with $p_g(S)=0$ such that $H^1(S, E)=0$, and $H$ is the tautological divisor of $\P(E)$.

\item $n=3, 4$: The first reduction $(\overline{X}, \overline{H})$ is either $(\P^3, \mathcal{O}_{\P^3}(3)), (\P^4, \mathcal{O}_{\P^4}(2)), (Q^3, \mathcal{O}_{\P^4}(2)|_{Q^3})$, or $(Y, L)$, where $Q^3 \subset \P^4$ is a three-dimensional quadric hypersurface, $Y$ is a three-dimensional scroll over a smooth projective curve $C$, and $L$ is a very ample divisor on $Y$ such that $H^1(Y, L)=0$ and it induces $\mathcal{O}_{\P^2}(2)$ on each fiber. In this case, $|K_X+(n-1)H|$ induces a morphism $\varphi \colon X \to \overline{X}$, which is a blow-up of $\overline{X}$ at finitely many distinct points, and $K_X + (n-1)H = \varphi^*(K_{\overline{X}} + (n-1)\overline{H})$.
\end{enumerate}
\end{theorem}

Note that del Pezzo manifolds were completely classified by Fujita \cite{F1}, \cite{F2}. For the cases $(3), (5), (6)$, one may want to further classify vector bundles $E$ with $H^1(E)=0$, but the complete solution to this problem seems to be out of reach.

For the next step of the classification of smooth projective varieties $X \subsetneq \P^r$ with $\reg(X)=3$, one needs to show the $2$-normality of projective varieties in Theorem \ref{classthmreg2}. It is already a very difficult problem even in the curve or the scroll case because we should consider projected varieties. There have been some studies of the $k$-normality of projected varieties (see e.g., \cite{AR}, \cite{KEP}).

\medskip

We now turn to syzygetic properties of smooth projective varieties $X \subsetneq \P^r$ with $\reg(\mathcal{O}_X)=2$. To study syzygies of projective varieties, it is often natural to impose the condition that $H^1(X, \mathcal{O}_X)=0$ (for instance, when we apply Green's duality theorem (\cite[Theorem (2.c.6)]{G}).

\begin{remark}
Among all seven cases in Theorem \ref{classthmreg2}, the varieties $X$ with $H^1(X, \mathcal{O}_X)=0$ can occur in all cases except $(1)$ and $(3)$. Actually, if a curve $C$ in $(5), (7)$ is rational or a surface $S$ in $(2), (6)$ is regular (i.e., $p_g(S)=q(S)=0$), then we have $H^1(X, \mathcal{O}_X)=0$.
\end{remark}

It would be very interesting to investigate the Koszul cohomology groups $K_{p,q}(X,V)$ and the Betti table of the section module $R(X, H)$ as a graded $\text{Sym}^\bullet (V)$-module.
Under the assumption that $\reg(\mathcal{O}_X)=2$ and $H^1(X, \mathcal{O}_X)=0$, we have $H^i(X, \mathcal{O}_X(k))=0$ for $0 < i < n$ and $k \in \Z$, and thus, the section ring $R(X, H):=\bigoplus_{m \geq 0} H^0(X, \mathcal{O}_X(mH))$ is Cohen-Macaulay. Furthermore, the syzygies of $R(X, H)$ is the same as those of the section ring $R(C, H|_C)$ of a general curve section $C$ of $X \subsetneq \P^r=\P(V)$.
We remark that $\reg(R(X, H))=\reg(\mathcal{O}_X)=2$ so that the Betti table has height $3$, i.e., $K_{p,q}(X, V)=0$ for $q\ge 3$.
For the basic notation of syzygies and Koszul cohomologies, see Subsection \ref{koscoh}.

There are several related results. It is easy to see that if $\reg(\mathcal{O}_X)=1$, then one can calculate all graded Betti numbers $k_{p,q}(X,V)$ from the Hilbert polynomial.
Ahn-Han \cite{AH} studied syzygies of homogeneous coordinate rings of projective schemes $X \subsetneq \P^r$ with $\reg(X)=3$.
Ein-Lazarsfeld \cite{EL1} gave a general picture of the asymptotic behavior of the Koszul cohomology groups $K_{p,q}(X, V)$ when $V=H^0(X, \mathcal{O}_X(1))$ and $\mathcal{O}_X(1)$ is sufficiently positive. In this case, $X \subsetneq \P^r=\P(V)$ satisfies $N_k$-property for some integer $k >0$.
Recall that $X \subsetneq \P(V)$ satisfies \emph{$N_0$-property} if it is projectively normal, and so the section ring $R(X, H)$ is the same as the homogeneous coordinate ring of $X \subsetneq \P(V)$. We say that $X \subsetneq \P(V)$ satisfies \emph{$N_k$-property} for some integer $k>0$ if it satisfies $N_0$-property and $K_{p,q}(X, V)=0$ for $1 \leq p \leq k $ and $q \geq 2$, i.e., the syzygy modules of $R(X, H)$ have only linear relations starting from quadrics up to $k$-step.

It is also an interesting problem to study the case when $X \subsetneq \P(V)$ is not linearly normal, i.e., $V \subsetneq H^0(X, \mathcal{O}_X(1))$.
The most promising case is when $X$ is a curve because Ein-Lazarsfeld's gonality theorem \cite{EL2} (see also \cite{R}) completely determines the vanishing and nonvanishing of the Koszul cohomology groups $K_{p,q}(X,V)$ for a complete embedding curve.
The second main result of the paper is the following:

\begin{theorem}\label{syzthm}
Let $X \subsetneq \P^r=\P(V)$ be a non-degenerate smooth projective variety of dimension $n$, codimension $e$, and degree $d$, and $H$ be its general hyperplane section. Suppose that $\reg(\mathcal{O}_X)=2$ and we further assume $H^1(X, \mathcal{O}_X)=0$ when $n \geq 2$.
Let $C$ be its general curve section of genus $g=h^0(X, \mathcal{O}_X(K_X + (n-1)H))$ and gonality $\gon(C)$.
\begin{enumerate}[\noindent$(1)$]
\item Suppose that $V = H^0(X, \mathcal{O}_X(1))$, i.e., $X \subsetneq \P^r$ is linearly normal.
If $e \geq g+k$ for some $k \geq 0$, then $X \subsetneq \P^r$ satisfies $N_k$-property, i.e., $K_{0,1}(X, V)=K_{p,q}(X, V)=0$ for $p \leq k$ and $q \geq 2$.

\item Suppose that $V \subsetneq H^0(X, \mathcal{O}_X(1))$ is of codimension $t \geq 1$ and $\mathcal{O}_X(1)$ is sufficiently positive.  For simplicity, we further assume that $e \geq g+1$. Then we have the following:
 \begin{enumerate}[$(i)$]
 \item $K_{p,q}(X, V)=0$ unless $0 \leq p \leq e$ and $q =0,1,2$.
 \item $K_{p,0}(X, V) \neq 0$ if and only if $p=0$.
 \item $K_{p,1}(X, V) \neq 0$ for $0 \leq p \leq e+1 - \gon(C)$ and $K_{p,1}(X, V) = 0$ for $e+2-\gon(C) + t \leq p \leq e$.
 \item $K_{p,2}(X, V) = 0 $ for $0 \leq p \leq e-1-g$ and $K_{p,2}(X, V) \neq 0$ for $e-g+t \leq p \leq e$.
\end{enumerate}
\end{enumerate}
\end{theorem}

The Betti table of the variety in Theorem \ref{syzthm} $(2)$ is as follows:
\smallskip
{\tiny \noindent\[
\begin{array}{c|c|c|c|c|c|c|c|c|c|c|c|c|c|c|c|c}
&  0 & 1  & \cdots & e-g-1 &  & \cdots & & e-g+t & \cdots & e+1-\gon(C) &  & \cdots &  & e+2-\gon(C)+t &  \cdots & e \\
\hline
0 & \textcolor{Blue}{1} & \textcolor{Green}{-} & \cdots & \textcolor{Green}{-} & \textcolor{Green}{-} & \cdots & \textcolor{Green}{-} & \textcolor{Green}{-} & \cdots & \textcolor{Green}{-}  & \textcolor{Green}{-} & \cdots & \textcolor{Green}{-} & \textcolor{Green}{-} & \cdots & \textcolor{Green}{-}\\
\hline
1 & \textcolor{Blue}{t}  & \textcolor{Blue}{*}  & \cdots & \textcolor{Blue}{*}  & \textcolor{Blue}{*}  & \cdots & \textcolor{Blue}{*}  & \textcolor{Blue}{*}  & \cdots & \textcolor{Blue}{*}  & \textcolor{Red}{?} & \cdots & \textcolor{Red}{?} & \textcolor{Green}{-} & \cdots & \textcolor{Green}{-}\\
\hline
2 & \textcolor{Green}{-} & \textcolor{Green}{-} & \cdots & \textcolor{Green}{-} & \textcolor{Red}{?} & \cdots & \textcolor{Red}{?} & \textcolor{Blue}{*}  & \cdots & \textcolor{Blue}{*}  & \textcolor{Blue}{*}  & \cdots & \textcolor{Blue}{*}  & \textcolor{Blue}{*}  & \cdots & \textcolor{Blue}{*} \\
\end{array}
\]}
\smallskip
Here $\textcolor{Green}{-}$, $\textcolor{Blue}{*}$, and $\textcolor{Red}{?}$ mean \textcolor{Green}{vanishing}, \textcolor{Blue}{non-vanishing}, and \textcolor{Red}{undetermined}, respectively.

In Remark \ref{middlerange}, we explain how to determine the undetermined Koszul cohomology groups $K_{p,1}(X, V)$ and $K_{p',2}(X, V)$ for the ranges $e+2-\gon(C) \leq p \leq e+1-\gon(C)+t$ and $e-g \leq p' \leq e-g+t-1$.
They heavily depend on the choice of $V \subset H^0(X, \mathcal{O}_X(1))$ and syzygies of $X \subsetneq \P^r$.

\medskip

Finally, we consider the Castelnuovo-Mumford regularity of a non-degenerate projective variety $X \subsetneq \P^r$ of degree $d$ and codimension $e$.
Eisenbud-Goto \cite{EG} conjectured that
$$
\reg(X) \leq d-e+1.
$$
This conjecture was verified for the curve case \cite{GLP} and the smooth surface case \cite{P}, \cite{L}.
Recently, McCullough-Peeva \cite{MP} constructed counterexamples to Eisenbud-Goto conjecture. However, many people still believe that the conjecture holds for smooth varieties (or mildly singular varieties). There are some partial results; see \cite{KJP} and the references therein.

In the case that $X \subsetneq \P^r$ is a non-degenerate smooth projective variety with $\reg(\mathcal{O}_X) \leq 2$ and $H^1(X, \mathcal{O}_X)=0$,  we can show a sharp bound for $\reg(X)$.
In \cite{N}, Noma proved a refined bound for $\reg(C)$, where $C \subset \P^r$ is a non-degenerate projective curve of degree $d$, codimension $e$, and arithmetic genus $g$. More precisely, \cite[Theorem 1]{N} says that if $C \subset \P^r$ is not linearly normal and $e \geq g+1$, then
$$
\reg(C) \leq d-e+1-g.
$$
The last main result of this paper is a generalization of the above result of Noma.

\begin{theorem}\label{regthm}
Let $X \subsetneq \P^r$ be a non-degenerate smooth projective variety of dimension $n$, codimension $e$, and degree $d$, and $H$ be its general hyperplane section. Suppose that $\mathcal{O}_X$ is $2$-regular and we further assume that $H^1(X, \mathcal{O}_X)=0$ when $n \geq 2$. Denote by $g:=h^0(X, \mathcal{O}_X(K_X + (n-1)H))$ the sectional genus of $X \subsetneq \P^r$.
If $X \subsetneq \P^r$ is not linearly normal and $e \geq g+1$, then
$$
\reg(X) \leq d-e+1-g.
$$
\end{theorem}

Note that when $X \subsetneq \P^r$ is linearly normal and $e \geq g$, it is projectively normal by Theorem \ref{syzthm}. Thus $\reg(X)=3 > d-e+1-g=2$.

We give a simple proof of Theorem \ref{regthm} using basic properties of syzygies of section rings in \cite{G} and Gruson-Lazarsfeld-Peskine technique in \cite{GLP}.
The key point is that partial information of minimal free resolution of the section ring controls the Castelnuovo-Mumford regularity, which is an invariant of the homogeneous coordinate ring.

If $\reg(\mathcal{O}_X)=1$, then $g=0$ and our regularity bound is exactly the conjectured one. This statement was previously shown in \cite[Theorem 5.2]{KEP}, \cite[Theorem 1.1]{NP}. There are many rational scrolls $X \subsetneq \P^r$ with $\reg(X)=d-e+1$.

\medskip

It is worth noting that a sharp upper bound for $\reg(\mathcal{O}_X)$ of a smooth projective variety $X\subset \P^r$ of arbitrary dimension was shown in \cite{KJP}. More precisely, if $X \subsetneq \P^r$  is a non-degenerate smooth projective variety of degree $d$ and codimension $e$, then
$$
1 \leq \reg(\mathcal{O}_X) \leq d-e.
$$
Furthermore, we have the classification of the extremal and next to extremal cases for both inequalities (see \cite[Theorem A]{KJP} for the upper bound cases).
We note that McCullough-Peeva \cite{MP} constructed infinitely many singular projective varieties $X \subsetneq \P^r$ such that $\reg(\mathcal{O}_X) > d-e$.
Although $\reg(X)$ has not been well understood yet in general, we have a fairly good understanding of $\reg(\mathcal{O}_X)$

\medskip
The organization of the remaining of the paper is as follows. Section \ref{clsec} is devoted to the proof of Theorem \ref{classthmreg2}, a classification of smooth projective varieties with $2$-regular structure sheaf using adjunction mappings. In Section \ref{syzsec}, we study the syzygy of such varieties, and we show Theorem \ref{syzthm}. Finally, in Section \ref{cmsec}, we give the proof of Theorem \ref{regthm}.

\section{Classification via adjunction mapping}\label{clsec}

In this section, we prove the first main result, Theorem \ref{classthmreg2}, using adjunction mappings. For basics of adjunction theory, we refer to \cite{BS}, \cite{Io1}.

First, we need the following.

\begin{lemma}\label{h1same}
Let $\overline{X}$ be a smooth projective variety, and $\varphi \colon X \to \overline{X}$ be a blow-up $\overline{X}$ at $k$ distinct points with the exceptional divisors $E_1, \ldots, E_k$. Let $\overline{H}$ be a very ample divisor on $\overline{X}$. Suppose that $H = \varphi^*\overline{H} - E_1 - \cdots -E_k$ is also very ample. Then we have
$H^1(X, \mathcal{O}_X(H))=H^1(\overline{X}, \mathcal{O}_{\overline{X}}(\overline{H}))$.
\end{lemma}

\begin{proof}
It is sufficient to show the assertion for $k=1$. Thus we assume that $\varphi \colon X \to \overline{X}$ is a blow-up at a point $x$ with the exceptional divisor $E=E_1$. We denote by $\mathcal{I}_x$ the ideal sheaf of $x$ in $X$.
It is a well-known fact that $\varphi_* \mathcal{O}_X(-E)=\mathcal{I}_{x}$ and $R^i \varphi_*\mathcal{O}_X(-E)=0$ for $i > 0$.
By the projection formula, we have 
$$
\text{ $\varphi_* \mathcal{O}_X(H) = \mathcal{O}_{\overline{X}}(\overline{H}) \otimes \mathcal{I}_x$ and $R^i \varphi_* \mathcal{O}_X(H)=0$ for $i > 0$. }
$$
Then the Leray spectral sequence yields
$$
H^i(X, \mathcal{O}_X(H))=H^i(\overline{X}, \mathcal{O}_{\overline{X}}(\overline{H}) \otimes \mathcal{I}_x) \text{ for any $i \geq 0$}.
$$
In particular, $H^1(X, \mathcal{O}_X(H))=H^1(\overline{X}, \mathcal{O}_{\overline{X}}(\overline{H}) \otimes \mathcal{I}_x)$.
Now consider the following short exact sequence
$$
0 \longrightarrow \mathcal{O}_{\overline{X}}(\overline{H}) \otimes \mathcal{I}_x  \longrightarrow \mathcal{O}_{\overline{X}}(\overline{H})  \longrightarrow \mathcal{O}_x \longrightarrow 0.
$$
Since $\overline{H}$ is very ample, $H^0(\overline{X}, \mathcal{O}_{\overline{X}}(\overline{H})) \to H^0(\overline{X}, \mathcal{O}_x)$ is surjective.
Note that $H^1(\overline{X}, \mathcal{O}_x)=0$. Thus
$$
H^1(\overline{X}, \mathcal{O}_{\overline{X}}(\overline{H})\otimes \mathcal{I}_x ) = H^1(\overline{X}, \mathcal{O}_{\overline{X}}(\overline{H}) ).
$$
Hence we are done.
\end{proof}

We now give the proof of Theorem \ref{classthmreg2}.

\begin{proof}[Proof of Theorem \ref{classthmreg2}]
Recall the setting: $X \subsetneq \P^r$ is a non-degenerate smooth projective variety of dimension $n$, and $H$ is its general hyperplane section.
If $n=1$, then $\mathcal{O}_X$ is $2$-regular if and only if $H^1(X, \mathcal{O}_X(1))=0$.
If $n \geq 2$, then by Kodaira vanishing theorem and Serre duality, $\mathcal{O}_X$ is $2$-regular if and only if
$$
H^1(X, \mathcal{O}_X(1))=H^2(X, \mathcal{O}_X)=H^0(X, \mathcal{O}_X(K_X + (n-2)H))=0.
$$

The direction $(\Leftarrow)$ is trivial except for $(7)$. If $X$ is a variety in $(7)$, we only have to check that $H^1(X, \mathcal{O}_X(1))=0$. This vanishing follows from Lemma \ref{h1same}.

Now, we show the direction $(\Rightarrow)$, so we assume that $\reg(\mathcal{O}_X)=2$.
If $n=1$ or $2$, then we immediately get $(1)$ or $(2)$, respectively.
Thus, from now on, we assume that $n \geq 3$.
Suppose first that $K_X+(n-1)H$ is not base point free. Then by \cite[Theorem 1.4]{Io1},
$$
(X, H)=(\P^n, \mathcal{O}_{\P^n}(1)), (Q^n, \mathcal{O}_{\P^{n+1}}(1)|_{Q^n}), (\P^2, \mathcal{O}_{\P^2}(2)), \text{or } (\P(E), \mathcal{O}_{\P(E)}(1)),
$$
where $Q \subset \P^{n+1}$ is a quadric hypersurface and $E$ is a very ample vector bundle of rank $n$ on a smooth projective curve $C$. In this case, $\reg(\mathcal{O}_X)=2$ implies that $X$ is a scroll over a smooth projective curve $C$ of genus $g \geq 1$ and $H^1(C, E)=0$. Thus we get $(3)$.

Suppose now that $K_X+(n-1)H$ is base point free. We can define the \emph{adjunction mapping}
$$\varphi \colon X \rightarrow B$$
given by $|K_X + (n-1)H|$. By {\cite[Proposition 1.11]{Io1}}, one of the following holds:
\begin{enumerate}[\indent$(a)$]
 \item $\dim B = 0$: $X$ is a del Pezzo manifold and $-K_X = (n-1)H$.
 \item $\dim B = 1$: $\varphi$ gives a hyperquadric fibration over a smooth projective curve $B$.
 \item $\dim B = 2$: $\varphi$ gives a linear fibration over a smooth projective surface $B$.
 \item $\dim B = n$: $\varphi$ is a birational morphism.
\end{enumerate}
In cases $(b), (c)$, $\varphi$ has connected fibers by \cite[Theorem 11.2.4]{BS}.
The cases $(a),(b),(c)$ correspond to the cases $(4),(5),(6)$, respectively. For the case $(b)$, we need to use \cite[Lemma 6]{Io2}.

It only remains to consider the case that the adjunction mapping $\varphi \colon X \to B$ is a birational contraction. Let $\overline{X}:=B$ and $\overline{H}$ be a divisor on $\overline{X}$ such that $\mathcal{O}_{\overline{X}}(\overline{H})=(\varphi_* \mathcal{O}_X(H))^{**}$. In this case, $(\overline{X}, \overline{H})$ is called the \emph{first reduction of $(X, H)$}, and the adjunction mapping $\varphi \colon X \to \overline{X}$ is a blow-up of $\overline{X}$ at finitely many distinct points (see \cite[Definition 7.3.3]{BS}).
By \cite[Corollary 7.4.2]{BS}, $\overline{H}$ is a very ample divisor.
Note that
$$
K_X + (n-1)H = \varphi^*(K_{\overline{X}} + (n-1)\overline{H})
$$
(see \cite[Theorem 7.3.2]{BS}). 
If $E_1, \ldots, E_k$ are exceptional divisors of the blow-up $\varphi \colon X \to \overline{X}$ at $k$ distinct points, then
$$
H = \varphi^*\overline{H} - \frac{1}{n-1}\left( K_X - \varphi^*K_{\overline{X}}\right)=\varphi^*\overline{H} - E_1 - \cdots - E_k.
$$
By \cite[Proposition 7.6.1]{BS}, we also have 
$$
H^0(X, \mathcal{O}_X(K_X + (n-2)H))=H^0(\overline{X}, \mathcal{O}_{\overline{X}}(K_{\overline{X} }+ (n-2)\overline{H})=0.
$$
By \cite[Theorem 11.7.1]{BS}, either
\begin{enumerate}[\indent$($i$)$]
 \item $(\overline{X}, \overline{H}) = (\P^3, \mathcal{O}_{\P^3}(3)), (\P^4, \mathcal{O}_{\P^4}(2)), (Q^3, \mathcal{O}_{\P^4}(2)|_{Q^3})$, or $(Y, L)$, where $Y$ is a three-dimensional scroll over a smooth projective curve such that $L$ induces $\mathcal{O}_{\P^2}(2)$ on each fiber, or
 \item $K_{\overline{X}} + (n-2)\overline{H}$ is semiample.
\end{enumerate}
In the case $($i$)$, we get $(7)$ by considering Lemma \ref{h1same}. In the case $($ii$)$, we have
$$
h^0(\overline{X}, \mathcal{O}_{\overline{X}}(K_{\overline{X} }+ (n-2)\overline{H})>0
$$
by \cite[Theorem 7.2.6]{BS} (alternatively, we can apply \cite[Proposition 13.2.4]{BS} to directly see that $h^0(X, \mathcal{O}_X(K_X + (n-2)H))>0$). Thus the case $($ii$)$ cannot occur when $\mathcal{O}_X$ is $2$-regular.
Therefore, we finish the proof.
\end{proof}

\section{Syzygies of section rings}\label{syzsec}

In this section, we study syzygies of section rings of smooth projective varieties with $2$-regular structure sheaf, and we show Theorem \ref{syzthm}.

\subsection{Koszul cohomology}\label{koscoh}
In this subsection, we recall notation of Koszul cohomology, and then show Theorem \ref{syzthm} (1). For more details on Koszul cohomology and syzygy, we refer to \cite{AN}, \cite{E}, \cite{EL1}, \cite{G}, \cite[Section 1.8]{positivityI}.

Let $X \subsetneq \P^r=\P(V)$ be a smooth projective variety, and $H$ be its general hyperplane section. Fix a divisor $B$ on $X$, and define a section module
$$
R=R(X, B, H):= \bigoplus_{m \in \Z} H^0(X,\mathcal{O}_X(B+ mH)).
$$
Then $R$ is naturally a graded $S=\text{Sym}^\bullet (V)$-module so that it has a minimal free resolution
$$
\cdots \to \bigoplus_{q} K_{p,q}(X, B, V) \otimes_k S(-p-q) \to \cdots  \to \bigoplus_{q} K_{0,q}(X, B, V) \otimes_k S(-q)  \to R \to 0,
$$
in which the vector space $K_{p,q}(X, B, V)$ is the \emph{Koszul cohomology group} associated to $B$ with respect to $H$.
It is well known that $K_{p,q}(X, B, V)$ is the cohomology of the Koszul-type complex
$$
\Lambda^{p+1}V \otimes H^0(X, B+(q-1)H) \to \Lambda^p V \otimes H^0(X, B+qH) \to \Lambda^{p-1}V \otimes H^0(X, B+(q+1)H).
$$
We put $k_{p,q}(X, B, V):=\dim_{\C} K_{p,q}(X, B, V)$. If $B=0$, we simply write $K_{p,q}(X, V)=K_{p,q}(X, 0, V)$. If $V=H^0(X, \mathcal{O}_X(H))$, then we write $K_{p,q}(X, V)=K_{p,q}(X, H)$.

The following is immediate from the definition.

\begin{proposition}\label{kosvan}
When $B=0$, we have $k_{0,0}(X,V) =1$ and $k_{p,0}(X, V)=0$ for $p \geq 1$. Furthermore, $k_{p,q}(X,V)=0$ unless $0 \leq p \leq \pr.\dim R(X, H)$ and $0 \leq q \leq \reg(\mathcal{O}_X)$.
\end{proposition}

We say that a smooth projective variety $X \subsetneq \P^r=\P(V)$ satisfies \emph{$N_0$-property} if $k_{0,q}(X, V)=0$ for $q \geq 1$ and it satisfies \emph{$N_k$-property} for some integer $k > 0$ if it satisfies $N_0$-property and $k_{p,q}(X, V)=0$ for $1 \leq p \leq k$ and $q \geq 2$. Note that  $X \subsetneq \P^r$ satisfies $N_0$-property if and only if projectively normal and it satisfies $N_1$-property if and only if it is projectively normal and its defining ideal is generated by quadrics.

\medskip

Now, we recall the setting of Theorem \ref{syzthm}: $X \subsetneq \P^r=\P(V)$ is a smooth projective variety of dimension $n$, codimension $e$, and degree $d$ such that $\reg(\mathcal{O}_X)=2$ and we further assume that $H^1(X, \mathcal{O}_X)=0$ when $n \geq 2$.
We suppose that $X \subsetneq \P^r$ is possibly obtained by an isomorphic projection from $X \subsetneq \P^N=\P(H^0(X, \mathcal{O}_X(1)))$.
Let $H$ be a general hyperplane section, and $C$ be a general curve section whose genus is denoted by $g$. We first show the following.

\begin{proposition}\label{secgenus}
Under the above notations, we have $g = h^0(X, \mathcal{O}_X(K_X + (n-1)H))$.
\end{proposition}

\begin{proof}
By the assumption and Kodaira vanishing theorem, 
$$
\text{ $H^0(X, \mathcal{O}_X(K_X + (n-2)H))=H^1(X, \mathcal{O}_X(K_X + (n-2)H))=0$ if $n \geq 2$.}
$$ 
Note that $H^1(H, \mathcal{O}_H)=0$ and $\reg(\mathcal{O}_H)=\reg(\mathcal{O}_X)=2$. (In general, we only have $\reg(\mathcal{O}_H) \leq \reg(\mathcal{O}_X)$ but in our case the equality holds.)
Thus we obtain
$$
H^0(X, \mathcal{O}_X(K_X + (n-1)H)) = H^0(H, \mathcal{O}_H(K_H + (n-2)H)) = \cdots = H^0(C, \mathcal{O}_C(K_C)),
$$
so the assertion follows.
\end{proof}

Since $\reg(\mathcal{O}_X)=2$ and $H^1(X, \mathcal{O}_X)=0$ if $n \geq 2$, it follows that $H^i(X, \mathcal{O}_X(m))=0$ for $1 \leq i \leq n-1$ and $m \in \Z$. Then the section ring $R(X, H)=\bigoplus_{m \geq 0} H^0(X, \mathcal{O}_X(mH))$ is Cohen-Macaulay.
Let $Y \in |H|$ be an irreducible smooth variety (when $n \geq 2$). Then $H^i(Y, \mathcal{O}_Y(mH|_Y))=0$ for $1 \leq i \leq n-2$ and $m \in \Z$. Thus $R(Y, H|_Y)$ is also Cohen-Macaulay. As a consequence, the section ring $R(C, H|_C)$ of a general curve section $C$ of $X \subsetneq \P^r$ is also Cohen-Macaulay. Furthermore, we have the following.

\begin{proposition}\label{samesyz}
Under the above notations, we have $K_{p,q}(X, H) = K_{p,q}(Y, H|_Y)$ for all $p,q$. In particular, $K_{p,q}(X, H)=K_{p,q}(C, H|_C)$ for all $p, q$.
\end{proposition}

\begin{proof}
Since $H^1(X, \mathcal{O}_X(m))=0$ for $m \in \Z$, the assertion follows from \cite[Theorem (3.b.7)]{G} (see also \cite[Theorem 2.21]{AN}).
\end{proof}

Note that the projective dimension of $R(X, H)$ is $e$, i.e., $K_{p,q}(X, H)=0$ for $p \geq e+1$.
We can also apply Green's duality theorem (\cite[Theorem (2.c.6)]{G}).

We now give the proof of Theorem \ref{syzthm} $(1)$.

\begin{proof}[Proof of Theorem \ref{syzthm} $(1)$]
For this case, we assume that $V = H^0(X, \mathcal{O}_X(1))$, i.e., $X \subsetneq \P^r$ is linearly normal. Since $H^1(\mathcal{O}_X)=0$, a general hyperplane section $Y:=X\cap H \subsetneq \P^{r-1}$ is also linearly normal and $H^1(Y, \mathcal{O}_{Y})=0$ if $\dim Y \ge 2$. Therefore, we conclude that a general curve section $C \subset \P^{e+1}$ is also linearly normal.
In particular, it is elementary to check that $\reg(\mathcal{O}_{X\cap H})\le \reg(\mathcal{O}_X)=2$, and thus, $H^1(C, \mathcal{O}_C(1))=0$. By Riemann-Roch formula,
$$
e+2=h^0(C, \mathcal{O}_C(1)) = d-g+1,
$$
so $d=e+g+1$. If $e \geq g+k$, then $d \geq 2g+1+k$. By Green's $(2g+1+k)$-Theorem (\cite[Theorem (4.a.1)]{G}), $C \subset \P^{e+1}$ satisfies $N_k$-property, and so does $X \subsetneq \P^r$.
\end{proof}

\subsection{Koszul cohomology under projection}\label{koszulproj}
In this subsection, we discuss about the effect of projection on Koszul cohomology (see \cite[Section 2.2]{AN} for more details), and then we prove Theorem \ref{syzthm} (2).

Let $X \subsetneq \P(V)$ be a smooth projective variety, and $H$ be its general hyperplane section. Consider an isomorphic projection from $X \subsetneq \P(V)$ to $\P(W)$ at one point $P \in \P(V)$. Then
$W \subset V$ is a subspace of codimension $1$, and we still have an embedding $X \subsetneq \P(W)$.
We can take a vector $v \in V^*$ which defines the point $P=[v] \in \P(V)$.
Then $\ev_v = \langle v, \cdot \rangle$ induces an exact sequence
$$
0 \longrightarrow W \longrightarrow V \xrightarrow{\ev_v} \C \longrightarrow 0.
$$
From the short exact sequence
$$
0 \longrightarrow \Lambda^p W \otimes R \longrightarrow \Lambda^p V \otimes R \longrightarrow \Lambda^{p-1} W \otimes R \longrightarrow 0,
$$
we obtain the following.

\begin{proposition}[{\cite[Lemma 2.9]{AN}}]\label{longseq}
We have the following long exact sequence
$$
K_{p,q}(X, B, W) \to K_{p,q}(X, B, V) \xrightarrow{\pr_v} K_{p-1,q}(X, B, W) \to K_{p-1, q+1}(X, B, W).
$$
\end{proposition}

\begin{corollary}\label{nonvan1}
Suppose that $K_{i,j}(X, B, W)=0$ for $i \geq p-1$ and $j \leq q-1$.
\begin{enumerate}[\noindent$(1)$]
\item If $K_{p,q}(X, B, V) \neq 0$, then $K_{p-1,q}(X, B, W) \neq 0$.
\item If $K_{p,q}(X, B, V) = 0$, then $K_{p,q}(X, B, W)=0$.
\end{enumerate}
\end{corollary}

\begin{proof}
By Proposition \ref{longseq} and the assumption, we have an exact sequence
$$
0 \to K_{p,q}(X, B, W) \to K_{p,q}(X, B, V) \to K_{p-1,q}(X, B, W).
$$
The assertion $(2)$ immediately follows. For $(1)$, suppose that $K_{p-1,q}(X, B, W)=0$. By the assumption, $K_{p,q}(X, B, W)=0$. Then we get a contradiction to $K_{p,q}(X, B, V) \neq 0$.
\end{proof}

On the other hand, we have a factorization
$$
\Lambda^p V \xrightarrow{\iota_v} \Lambda^{p-1}W \hookrightarrow \Lambda^{p-1} V,
$$
where $\iota_v$ is explicitly given as $\iota_v(v_1 \wedge \cdots \wedge v_p) = \sum_i (-1)^i v_1 \wedge \cdots \wedge \hat{v_i} \wedge \cdots \wedge v_p \otimes \ev_v(v_i)$.
Then we obtain a commutative diagram
\begin{equation}\label{comdia}
\xymatrix{
K_{p,q}(X, B, V)\ar[rd]_-{\ev_v} \ar[r]^-{\pr_v} & K_{p-1,q}(X, B, W) \ar[d]\\
 & K_{p-1,q}(X, B, V).
}
\end{equation}
As $v$ varies, the maps $\ev_v$ glue together to a homomorphism of vector bundles on $\P(V)$. Thus we obtain a map
$$
\ev \colon \mathcal{O}_{\P(V)} \otimes K_{p,q}(X, B, V) \to \mathcal{O}_{\P(V)}(1) \otimes K_{p-1,q}(X, B, V),
$$
and so we get a natural map
$$
H^0(\ev) \colon  K_{p,q}(X, B, V) \to V \otimes K_{p-1, q}(X, B, V).
$$
Let $x_0, \ldots, x_r$ be homogeneous coordinates of $\P(V)=\P^r$. Then $H^0(\ev)$ can be regarded as a $k_{p,q}(X, B, V) \times k_{p-1,q}(X, B, V)$ matrix whose entries are linear forms in $x_0, \ldots, x_r$.

\begin{proposition}[{\cite[Proposition 2.11]{AN}}]\label{evinj1}
If $H^0(X, \mathcal{O}_X(B-H))=0$ and $K_{p,0}(X, B, V)=0$ for $p \geq 1$, then
$$
H^0(\ev) \colon K_{p, 1}(X, B, V) \to V \otimes K_{p-1,q}(X, B, V)
$$
is injective for all $p \geq 2$
\end{proposition}

\begin{corollary}\label{van1}
Under the same assumptions in Proposition \ref{evinj1}, we have the following:
\begin{enumerate}[\noindent$(1)$]
\item For $p \geq 2$, if $k_{p-1,1}(X, B, V) < k_{p,1}(X, B, V)$, then $K_{p,1}(X, B, W) \neq 0$.
\item For $p \geq 2$, if $k_{p-1,1}(X, B, V) \geq k_{p,1}(X, B, V)>0$ and $W \subset V$ is general, then $K_{p,1}(X, B, W) =0$.
\end{enumerate}
\end{corollary}

\begin{proof}
By Proposition \ref{longseq} and the assumption, we have an exact sequence
$$
0 \to K_{p,q}(X, B, W) \to K_{p,q}(X, B, V) \xrightarrow{\pr_v} K_{p-1,q}(X, B, W).
$$
We see that 
$$
\text{ $K_{p,1}(X, B, W) = 0$ if and only if $\pr_v \colon K_{p,q}(X, B, V) \to K_{p-1,q}(X, B, W)$ is injective.}
$$
Note that $K_{p-1,q}(X, B, W) \to K_{p-1,q}(X, B, V)$ is injective. By considering the diagram (\ref{comdia}), we see that 
$$
\text{ $K_{p,1}(X, B, W) = 0$ if and only if $\ev_v \colon K_{p,q}(X, B, V) \to K_{p-1,q}(X, B, V)$ is injective. }
$$
The assertion $(1)$ then immediately follows. For the assertion $(2)$, we  regard $H^0(\ev)$ as a $k_{p,q}(X, B, V) \times k_{p-1,q}(X, B, V)$ matrix. By Proposition \ref{evinj1}, the maximal minors of $H^0(\ev)$ defines a proper closed subset of $\P(V)$. Since $W \subset V$ is general, we can assume that $[v]=P \in \P(V)$ is a general point. Thus $\ev_v$ is also injective, and this shows $(2)$.
\end{proof}

We also have the following.

\begin{proposition}\label{evinj2}
If $H^0(X, \mathcal{O}_X(B-H))=0$, then
$$
H^0(\ev) \colon K_{p,0}(X, B, V) \to V \otimes K_{p-1,0}(X, B, V)
$$
is injective for all $p \geq 1$.
\end{proposition}

\begin{proof}
We have a natural injective map $\iota \colon \Lambda^p V \to V \otimes  \Lambda^{p-1}V $ which is the dual of the wedge product map $\Lambda^{p-1} V^* \otimes V^* \to \Lambda^p V^*$. Explicitly, 
$$
\iota(v_1 \wedge \cdots \wedge v_p) = (v \mapsto \iota_v (v_1 \wedge \cdots v_p)) \in \text{Hom}(V^*, \Lambda^{p-1}V) = V \otimes  \Lambda^{p-1}V.
$$
Note that $K_{p,0}(X, B, V)$ is the kernel of the map
$$
\Lambda^p V \otimes H^0(X, B) \to \Lambda^{p-1}V \otimes H^0(X, B+H).
$$
From the following diagram
\[
\xymatrix{
K_{p,0}(X, B, V) \ar[d]_-{H^0(\ev)}  \ar@{^{(}->}[r]  & \Lambda^p V \otimes H^0(X, B) \ar@{^{(}->}[d]^-{\iota \otimes \text{id}}\\
V \otimes K_{p-1,0}(X, B, V)   \ar@{^{(}->}[r] & V \otimes  \Lambda^{p-1} V \otimes H^0(X, B),
}
\]
we see that $H^0(\ev)$ is injective.
\end{proof}

\begin{corollary}\label{van2}
Under the same assumptions in Proposition \ref{evinj2}, we have the following:
\begin{enumerate}[\noindent$(1)$]
\item For $p \geq 1$, if $k_{p-1,0}(X, B, V) < k_{p,0}(X, B, V)$, then $K_{p,0}(X, B, W) \neq 0$.
\item For $p \geq 1$, if $k_{p-1,0}(X, B, V) \geq k_{p,0}(X, B, V)>0$ and $W \subset V$ is general, then $K_{p,0}(X, B, W) =0$.
\end{enumerate}
\end{corollary}

\begin{proof}
The proof is identical to that of Corollary \ref{van1}, so we omit the detail.
\end{proof}

\smallskip

We now turn to the proof of Theorem \ref{syzthm} $(2)$.
By our assumptions, the syzygies of $R(X, H)$ is the same to those of the section ring $R(C, H|_C)$. Thus we assume that $X=C$ is a curve.
By Theorem \ref{syzthm} $(1)$ and Ein-Lazarsfeld's gonality theorem \cite{EL2} (see also \cite{R}), if $V = H^0(C, \mathcal{O}_C(1))$, then
\begin{enumerate}[\indent$(1)$]
  \item $K_{p,0}(C, V) \neq 0$ if and only if $p=0$.
  \item $K_{p,1}(C, V) \neq 0$ if and only if $1 \leq p \leq h^0(C, \mathcal{O}_C(1))-1-\gon(C)$.
  \item $K_{p,2}(C, V) \neq 0$ if and only if $h^0(C, \mathcal{O}_C(1))-1-g \leq p \leq h^0(C, \mathcal{O}_C(1))-2$.
\end{enumerate}
Thus the Betti table of $C \subset \P^N$, where $N=h^0(C, \mathcal{O}_C(1))-1$ is as follows:

\medskip

{\small \noindent\[
\begin{array}{c|c|c|c|c|c|c|c|c|c|c}
&  0 & 1 & \cdots & N-g-1 & N-g & \cdots & N-\gon(C) & N-\gon(C)+1 & \cdots & N-1 \\
\hline
0 & 1 & - & \cdots & - & - & \cdots & - & - & \cdots & -\\
\hline
1 & - & * & \cdots & * & * & \cdots & * & - & \cdots &- \\
\hline
2 & - & - & \cdots & - & * & \cdots & * & * & \cdots &*
\end{array}
\]}

\medskip

Here $-$ and $*$ mean vanishing and non-vanishing, respectively.

\begin{remark}\label{middlerange}
We explain how to determine vanishing or nonvanishing of Koszul cohomology groups in the setting of Theorem \ref{syzthm}. We only have to consider $K_{p,1}$ and $K_{p,2}$. Note that we know the Betti table of $X \subsetneq \P(H^0(X, \mathcal{O}_X(1))$.
Thus it is enough to determine the Betti table of $X \subsetneq \P(W)=\P^{n+e}$, where $W \subset V$ is a subspace of codimension $1$, under the assumption that  the Betti table of $X \subsetneq \P(V)=\P^{n+e+1}$ is given.

First, we consider $K_{p,1}$. There exists an integer $k>0$ such that
$$
K_{p,1}(X, V) \neq 0 \text{ for $1 \leq p \leq k$ and } K_{p,1}(X, V)=0 \text{ for $k+1 \leq p \leq e+1$}.
$$
If $V \neq H^0(X, \mathcal{O}_X(1))$, then $K_{0,1}(X, V) \neq 0$.
It follows from Corollary \ref{nonvan1} that
$$
K_{p,1}(X, W) \neq 0 \text{ for $0 \leq p \leq k-1$ and } K_{p,1}(X, W)=0 \text{ for $k+1 \leq p \leq e$}.
$$
For determining $K_{k,1}(X, W)$, we apply Corollary \ref{van1}. If $k_{k-1,1}(X, V) < k_{k,1}(X, V)$, then $K_{k,1}(X, W) \neq 0$. If $k_{k-1,1}(X, V) \geq k_{k,1}(X, V)$ and $W \subset V$ is furthermore general, then $K_{k,1}(X, W)=0$.

For $K_{p,2}$, we first apply Green's duality theorem (\cite[Theorem (2.c.6)]{G}) to the curve section $C \subset \P(V|_C)$, so we get $K_{p,2}(C, V|_C) = K_{e+1-p, 0}(C, K_C,V|_C)^*$. We also have $K_{p,2}(C, W|_C) = K_{e-p, 0}(C, K_C, W|_C)^*$. Then we can similarly argue as in the $K_{p,1}$ case using Corollaries \ref{nonvan1} and \ref{van2}.
\end{remark}

\begin{remark}
We see in Remark \ref{middlerange} that if $V \subsetneq H^0(X, \mathcal{O}_X(1))$ a general subspace of codimension $1$ and 
$$k_{h^0(C, \mathcal{O}_C(1))-1-g, 2}(X, H^0(X, \mathcal{O}_X(1))) \leq  k_{h^0(C, \mathcal{O}_C(1))-g, 2}(X, H^0(X, \mathcal{O}_X(1))),$$ 
then $K_{e-g,2}(X, V) = 0$. On the other hand, if there exists a section $s \in H^0(C, H|_C - K_C)$ such that $f_1 \cdot s, \ldots, f_g \cdot s \in V|_C$, where $f_1, \ldots, f_g$ are linearly independent sections of $H^0(C, K_C)$, then the argument in \cite[Proof of Proposition 5.1]{EL1} shows that $K_{e-g,2}(X, V) \neq 0$
\end{remark}

We are ready to prove Theorem \ref{syzthm} $(2)$.

\begin{proof}[Proof of Theorem \ref{syzthm} $(2)$]
We assume that $V \subsetneq H^0(X, \mathcal{O}_X(1))$ is of codimension $t \geq 1$.

\noindent $(i), (ii)$ These assertions follow from Proposition \ref{kosvan}.

\noindent $(iii), (iv)$ The assertions follow from the arguments in Remark \ref{middlerange}.
\end{proof}

\section{Castelnuovo-Mumford regularity}\label{cmsec}

In this section, we give a proof of Theorem \ref{regthm}.
Recall the setting of Theorem \ref{regthm}: $X \subsetneq \P^r$ is a non-degenerate smooth projective variety of dimension $n$, codimension $e$, and degree $d$, and $H$ is its general hyperplane section. We assume that $\mathcal{O}_X$ is $2$-regular and we further assume that $H^1(X, \mathcal{O}_X)=0$ when $n \geq 2$. Denote by $g:=h^0(X, \mathcal{O}_X(K_X + (n-1)H))$ the sectional genus of $X \subsetneq \P^r$ (see Proposition \ref{secgenus}).

For the proof, we apply Green's vanishing theorem (\cite[Theorem (3.a.1)]{G}) and Gruson-Lazarsfeld-Peskine technique in \cite{GLP}. 

\begin{proof}[Proof of Theorem \ref{regthm}]
We suppose that $X \subsetneq \P^r$ is not linearly normal and $e \geq  g+1$. Then 
$$
g=h^0(X, \mathcal{O}_X(K_X + (n-1)H)) \leq e-k
$$ 
for $k=0,1$. By Green's vanishing theorem (\cite[Theorem (3.a.1)]{G}), we obtain $K_{k,2}(X,V)=0$ for $k=0,1$. We can easily check that $k_{0,1}(X,V)=d-e-g-1$. Thus the sheafification of a minimal free resolution of $R(X, H)$ is of the form
$$
\cdots \longrightarrow \mathcal{O}_{\P^r}(-2)^{l} \longrightarrow \mathcal{O}_{\P^r}(-1)^{d-e-g-1} \oplus \mathcal{O}_{\P^r} \longrightarrow \mathcal{O}_X \longrightarrow 0.
$$
As in \cite[Proof of Theorem 2.1]{GLP}, we apply snake lemma to obtain the following commutative diagram
\[
\xymatrix{
 & & & 0 \ar[d] & & \\
 & 0 \ar[d] &  & \mathcal{O}_{\P^r}  \ar[r] \ar[d] & \mathcal{O}_X \ar[r] \ar@{=}[d] & 0 \\
0 \ar[r] & K \ar[r] \ar[d] & \mathcal{O}_{\P^r}(-2)^{l} \ar[r]\ar@{=}[d] & \mathcal{O}_{\P^r}(-1)^{d-e-g-1} \oplus \mathcal{O}_{\P^r} \ar[r] \ar[d] & \mathcal{O}_X \ar[r] & 0\\
0 \ar[r] & N \ar[r] \ar[d] & \mathcal{O}_{\P^r}(-2)^{l} \ar[r]^-{u} & \mathcal{O}_{\P^r}(-1)^{d-e-g-1} \ar[r] \ar[d] & 0 & \\
& \mathcal{I}_{X|\P^r} \ar[d] & & 0 &  & \\
& 0 &&&&
}
\]
in which all horizontal and vertical sequences are exact. Note that 
$$
H^1(\P^r, K(m))=H^2(\P^r, K(m))=0 \text{  for all $m \in \Z$.}
$$ 
Thus we have
$$
H^1(\P^r, N(m))=H^1(\P^r, \mathcal{I}_{X|\P^r}(m))\text{  for all $m \in \Z$.}
$$
By considering the Eagon-Northcott complex associated to $u$, we see that $N$ is $(d-e-g+1)$-regular (see e.g., \cite[Lemma 5]{N}). Therefore, $H^1(\P^r, \mathcal{I}_{X|\P^r}(d-e-g))=0$, so we finally obtain
$\reg(X) \leq d-e+1-g$.
\end{proof}

$ $

\end{document}